\documentclass[11pt]{amsart}
\usepackage{amsmath,amssymb}
\numberwithin{equation}{section}
 \usepackage{xcolor}

\newtheorem{theorem}{Theorem}[section]
\newtheorem{thm}[theorem]{Theorem}

\newtheorem{lem}[theorem]{Lemma}
\newtheorem{prop}[theorem]{Proposition}
\newtheorem{coro}[theorem]{Corollary}

\theoremstyle{definition}

\newtheorem{setting}[theorem]{Setting}
\newtheorem{defi}[theorem]{Definition}

\newtheorem{exa}[theorem]{Example}

\newtheorem{rem}[theorem]{Remark}
 \newtheorem*{ackn}{Acknowledgements}
 
 \newtheorem*{thmA}{Theorem A} 
 \newtheorem*{thmB}{Theorem B} 
\newtheorem*{thmC}{Theorem C}

 \theoremstyle{plain}
\newtheorem*{namedthm}{\namedthmname}
\newcounter{namedthm}

\makeatletter

 \newcommand{\D}{\mathbb D}
 \newcommand{\R}{\mathbb R}
 \newcommand{\Q}{\mathbb Q}
 \newcommand{\C}{\mathbb C}
  \newcommand{\PP}{\mathbb P}
 \newcommand{\N}{\mathbb N}

 \newcommand{\e}{\varepsilon}

 \newcommand{\f}{\varphi}
 
 \newcommand{\p}{\psi}

 \newcommand \PSH {{\rm PSH}}

\newcommand{\Ric}{{\rm Ric}}

 \usepackage{hyperref}
\hypersetup{
    unicode=false,        
    pdftoolbar=true,      
    pdfmenubar=true,       
    pdffitwindow=false,     
    pdfstartview={FitH},    
    pdftitle={Diameter estimates},    
    pdfauthor={Guedj, To},     
    colorlinks=true,       
   linkcolor=blue,          
    citecolor=blue,        
    filecolor=black,      
    urlcolor=blue}



\begin{document}

\title[K\"ahler families of Green's functions]{K\"ahler families of Green's functions}

\author{Vincent Guedj  \&  Tat Dat T\^o}

\address{Institut Universitaire de France et Institut de Math\'ematiques de Toulouse   \\ Universit\'e de Toulouse \\
118 route de Narbonne \\
31400 Toulouse, France\\}

\email{\href{mailto:vincent.guedj@math.univ-toulouse.fr}{vincent.guedj@math.univ-toulouse.fr}}
\urladdr{\href{https://www.math.univ-toulouse.fr/~guedj}{https://www.math.univ-toulouse.fr/~guedj/}}

\address{Sorbonne Universit\'e,  Institut de mathématiques de Jussieu – Paris Rive Gauche, 4, place Jussieu, 75252 Paris Cedex 05 France.}

\email{\href{mailto:tat-dat.to@imj-prg.fr}{tat-dat.to@imj-prg.fr}}
\urladdr{\href{https://sites.google.com/site/totatdatmath/home}{https://sites.google.com/site/totatdatmath/home}}

\date{\today}

 \begin{abstract}
 In a remarkable series of works, Guo, Phong, Song, and Sturm have obtained key uniform estimates for the Green's functions associated with certain K\"ahler metrics. In this note, we broaden the scope of their techniques by removing one of their assumptions and allowing the complex structure to vary. We apply our results to various families of canonical K\"ahler metrics.
 \end{abstract}

 \maketitle

\setcounter{tocdepth}{1}
\tableofcontents

\section*{Introduction}

Let $(X,\omega)$ be a compact K\"ahler manifold of complex dimension $n$.
We set $V_{\omega}=\int_X \omega^n$ and
let  $G_x^{\omega}$ denote the Green's function of $\omega$ at the point $x$.
This is the unique quasi-subharmonic function such that $\int_X G_x^{\omega} \omega^n=0$ and
$$
\frac{1}{V_{\omega}} (\omega+dd^c G_x^{\omega}) \wedge \omega^{n-1}=\delta_x.
$$
Here $\delta_x$ denotes the Dirac mass at point $x$. Equivalently
$
\Delta_{\omega} G_x^{\omega}=n \left( V_{\omega} \delta_x -\omega^n \right).
$

\smallskip

In a  remarkable series of articles \cite{GPS22,GPSS22,GPSS23}, Guo, Phong, Song and Sturm have 
obtained  several key estimates for $G_x^{\omega}$ that are uniform when $\omega$ varies 
in a large family ${\mathcal F}$ of K\"ahler forms
(our normalization for $G_x^{\omega}$ slightly differs from theirs, see Definition \ref{def:green} below).
 Although $G_x^{\omega}$ is a solution of the Laplace equation, its dependence on $\omega$ is highly non-linear.
 Nonetheless, these authors succeeded in obtaining spectacular estimates by comparing the problem at hand with an auxiliary 
complex Monge-Amp\`ere equation (an idea originating from \cite{CC21}),
for which fine uniform estimates are available 
(see \cite{Yau78,Kol98,EGZ09,EGZ08,DP10,BEGZ10,GL21,GPT23}).

\smallskip

Fix $\beta$ a reference K\"ahler form on $X$,
$A,B>0$ positive constants, and $\gamma \geq 0$ a continuous function such that $(\gamma =0)$  has small Hausdorff dimension.
The family ${\mathcal F}$ consists in those K\"ahler forms $\omega$ 
which satisfy the following three assumptions:
\begin{enumerate}
\item an upper-bound on the cohomology class $\int_X \omega \wedge \beta^{n-1} \leq A$;
\item a uniform pointwise lower bound $f_{\omega} \geq \gamma$, where
$f_{\omega}:=V_{\omega}^{-1}\omega^n/\beta^n$;
\item  a uniform upper bound  $\int_X f_{\omega} (\log f_{\omega})^p \beta^n \leq B$, where $p>n$.
\end{enumerate}

The first assumption can be restated as the cohomology class of $\omega$ remaining within a bounded subset of $H^{1,1}(X,\R)$.
It implies a uniform upper bound on the volume $V_{\omega} \leq C(A)$, but it is a strictly stronger assumption
(see Example \ref{exa:voldiam}).

The third assumption ensures that $\omega$ admits uniformly bounded potentials
(as shown in \cite{Kol98,EGZ08,DP10}), that moreover
have good continuity properties (see \cite{Kol08,DDGHKZ14,GPTW21,GGZ23}). It can be slightly generalized to
$$
\int_X f_{\omega} (\log f)^n (\log \log [3+f_{\omega}])^p \beta^n \leq B
\; \; \text{ with } \; \; 
p>2n,
$$
as shown in \cite{GGZ23}. In this article we rather stick to the more classical assumption
$\int_X f_{\omega}^p \beta^n \leq B$ with $p>1$, to simplify the exposition.

\medskip

Our goal in this note is twofold. 
First, we eliminate the second assumption by establishing a key new estimate (Proposition \ref{prop:key}).
    Second, we demonstrate that this new estimate, as well as all the estimates obtained in \cite{GPS22,GPSS22,GPSS23},
   remain uniform as the complex structure varies.
   This significantly broadens the scope of geometric applications for these results.

The precise setting is as follows. 
Let $\mathcal{X}$ be an irreducible and reduced  complex space. We let $\pi:{\mathcal X} \rightarrow 2 \mathbb D$ 
denote a proper, surjective holomorphic map with connected fibers such that each fiber $X_t=\pi^{-1}(t)$ is a $n$-dimensional 
 compact K\"ahler manifold, for $t\in 2 \mathbb D^*$. We allow the central fiber $X_0$ to be a singular,    irreducible and reduced complex space, as this is important for geometric applications.  
Here $\mathbb D$ is the unit disk in $\C$; we assume that the family is defined over a slightly larger disk
$2 \mathbb D$, and will establish estimates that are uniform with respect to the parameter $t \in \mathbb D^*$.

We fix $\beta$ a relative K\"ahler form on $\mathcal{X}$, i.e $\beta$ is a smooth  form on $\mathcal{X}$ whose restrictions  $\beta_t=\beta_{|X_t}$ are K\"ahler forms.  We assume that the volumes $V_{\beta_t}:=\int_{X_t} \beta_t^n$ are uniformly bounded away from 0 and $+ \infty $, and let $dV_{X_t}=\beta_t^n /V_{\beta_t}$  denote the corresponding  probability volume form on $X_t$.  

\smallskip

Fix $p>1$ and $A,B>0$.
We let ${\mathcal K}(\mathcal{X},p,A,B)$ denote the set of all relative K\"ahler forms $\omega$ on 
$\mathcal{X}\setminus X_0  $ such that for all $t \in \mathbb D^*$,  $\{\omega_t\}\leq A\{\beta_t\}$ in $H^{1,1}(X_t,\R)$, and
$$
\int_{X_t} f_t^p dV_{X_t} \leq B,
\; \; \text{ where } \; \; 
\frac{1}{V_{\omega_t}} \omega_t^n=f_t dV_{X_t}.
$$

Our main result is the following uniform control on Green's functions $G_x^{\omega_t}$.
 
 \begin{thmA}\label{thm_A}
    Fix $0<r<\frac{n}{n-1}$ and $0<s< \frac{2n}{2n-1}$.
For all $t \in \D^*, x \in X_t$ and  $\omega \in {\mathcal K}(\mathcal{X},p,A,B)$ we have
\begin{enumerate}
\item $\sup_{X_t} G_x^{\omega_t} \leq C_0=C_0(n,p,A,B)$;

\smallskip

\item $\frac{1}{V_{\omega_t}} \int_{X_t} |G_x^{\omega_t}|^r  \omega_t^n  \leq C_1=C_1(n,p,r,A,B)$;

\smallskip

\item $\frac{1}{V_{\omega_t}} \int_{X_t} |\nabla G_x^{\omega_t}|_{\omega_t}^s \omega_t^n \leq C_2=C_2(n,p,s,A,B)$.
\end{enumerate}
\end{thmA}

These estimates extend \cite[Theorem 1 and Theorem 2]{GPS22} and \cite[Theorem 1.1]{GPSS22}.
The uniformity with respect to the complex structure crucially depends on the uniform estimates
obtained in \cite{DnGG23}.

As in \cite{GPSS22} they easily imply a uniform control on diameters, as well
as a uniform non-collapsing estimate. Together with the uniform upper bound on volumes, this yields
in particular that the metric spaces $(X_t,\omega_t)$ are relatively compact in the Gromov-Hausdorff topology.

 \begin{thmB}
 Fix $0<\delta<1$, and $\omega \in {\mathcal K}(\mathcal{X},p,A,B)$. 
  There exists $C=C(n,p,\delta,A,B)>0$ such that  for all $ t\in \D^*, x \in X_t$ and $r>0$,
  $$
  C  \min(1,r^{2n+\delta}) \leq \frac{1}{V_{\omega_t}} {\rm Vol}_{\omega_t}(B_{\omega_t}(x,r))
  \; \; \text{ and } \; \; 
  {\rm diam}(X_t,\omega_t) \leq C.
  $$
 Thus the family of compact metric spaces 
  $\{ (X_t,d_{\omega_t}), \; \omega \in {\mathcal K}(\mathcal{X},p,A,B), \,  t \in \D^* \}$ 
  is pre-compact in the Gromov-Hausdorff topology.
\end{thmB}
 
Following \cite{GPSS23}, we further show in Theorem \ref{thm:sobolev}
that the Sobolev constants of the metrics $\omega_t$ are also uniformly bounded.

\smallskip

Theorems A and B have many geometric applications. We 
explain in Section \ref{sec:csck} how they provide uniform estimates
for families of constant scalar curvature K\"ahler metrics (see Corollary \ref{cor:csck}), following
\cite{CC21} and \cite{PTT23}.

As in \cite{GPSS22}, we then apply our estimates to the study of the K\"ahler-Ricci flow.
Assume that $X$ is a  a  smooth minimal model ($K_X$ nef), and consider
   $$
   \frac{\partial \omega_t}{\partial t}=-{\rm Ric}(\omega_t)-\omega_t
   $$
   the normalized K\"ahler-Ricci flow starting from an initial K\"ahler metric $\omega_0$.
   It follows from \cite{TZ06} that this   flow exists for all times $t>0$.
    We obtain here the following striking result which -in particular- solves \cite[Conjecture 6.2]{Tos18}:

 \begin{thmC}
   Fix $0<\delta<1$.
   There exist $C=C(\omega_0)>0$ and $c=c(\omega_0,\delta)>0$ such that for all $t>0$ and $x \in X$,
   $$
   {\rm diam}(X,\omega_t) \leq C  
   \; \; \text{ and } \; \;
     {\rm Vol}_{\omega_t}(B_{\omega_t}(x,r)) \geq c  r^{2n+\delta} V_{\omega_t},
   $$
   whenever $0<r<{\rm diam}(X,\omega_t)$.
 \end{thmC}
 
 This result has been established by Guo-Phong-Song-Sturm under the extra assumption that 
  the Kodaira dimension ${\rm kod}(X)$ of $X$ is non-negative, see \cite[Theorem 2.3]{GPSS22}.
  It is an old conjecture of Mumford that $K_X$ nef implies ${\rm kod}(X) \geq 0$.
  The latter is a weak form of the abundance conjecture, one of the most important open problem in 
  birational geometry.
   We refer the reader to \cite{Tos18,Tos24} for recent overviews of the theory of the K\"ahler-Ricci flow.

 \smallskip
 
 In Section \ref{sec:ke} we show how our estimates provide an alternative proof of an important
 diameter bound of Y.Li \cite[Theorem 1.4]{Li23}.
 We anticipate many more applications of Theorems A and B (see \cite{CGNPPW23,GP24} for recent developments).
  In \cite{GPSS23}, the authors have partially extended their techniques
to the case of mildly singular K\"ahler varieties, showing in particular that K\"ahler-Einstein currents have
bounded diameter. Our key Proposition \ref{prop:key} can be extended to this context as well, as it relies solely
on uniform estimates for solutions to complex Monge-Amp\`ere equations, 
which are applicable here according to \cite[Theorem 1.9]{DnGG23}.
Therefore, we expect that most of the results presented in this note can be extended to families of 
 K\"ahler varieties (see \cite[Section 3]{GP24} for some first steps).
  Very recently Vu \cite{V24} announced a proof of Theorem B for a fixed variety
  and an independent proof of Theorem C by using analysis on Sobolev spaces associated to currents;
shortly afterwards Guo-Phong-Song-Sturm \cite{GPSS24} have provided and alternative proof of Proposition \ref{prop:key}.

\begin{ackn} 
We thank D.H.Phong for enlightning discussions about his joint works with Guo, Song and Sturm on Green's functions.
We  are grateful to H.C.Lu for several useful comments on a preliminary draft.
T.D.T. is partially supported by the project MARGE ANR-21-CE40-0011.
V.G. is partially supported by the Institut Universitaire de France
and fondation Charles Defforey.
\end{ackn}


\section{Uniform estimates for Monge-Amp\`ere potentials}

 \subsection{Quasi-(pluri)subharmonic functions}
 
Let $(X, \omega)$ be a compact K\"ahler manifold of complex dimension $n$.

  \subsubsection{$\omega$-subharmonic functions}

 \begin{defi}
A function $v:X \rightarrow \R \cup \{-\infty\}$ 
is $\omega$-subharmonic ($\omega$-sh for short) if it is locally the sum of a smooth and a subharmonic function, and  
 $(\omega+dd^c u) \wedge \omega^{n-1} \geq 0$ in the sense of distributions.
 Alternatively 
 $$
 \Delta_{\omega} v=n \frac{dd^c v \wedge \omega^{n-1}}{\omega^n} \geq -n.
 $$
 \end{defi}
 
 The maximum principle ensures that two $\omega$-sh functions $u,v$ satisfy $\Delta_{\omega} u=\Delta_{\omega} v$ if and only if they differ by a constant. We can thus ensure that $u=v$ by normalizing them
by $\int_X u \omega^n=\int_X v \omega^n=0$.

\begin{defi} \label{def:green}
We let  $G_x$ denote the Green function of $\omega$ at the point $x$.
This is the unique $\omega$-sh function such that $\int_X G_x \omega^n=0$ and
$$
\frac{1}{V_{\omega}} (\omega+dd^c G_x) \wedge \omega^{n-1}=\delta_x.
$$
Here $\delta_x$ denotes the Dirac mass at point $x$.
\end{defi}

 Let us stress that our definition differs from that 
of \cite{GPSS22} in two ways: we use the opposite sign convention, as well as a different volume normalization.
If $\tilde{G}_x$ denotes the Green function from \cite{GPSS22}, then
$G_x=-\frac{V_{\omega}}{n}\tilde{G}_x$.

\smallskip

For any $\omega$-sh function $u$ such that $\int_X u \omega^n=0$, Stokes theorem yields
$$
u(x)=\frac{1}{V_{\omega}} \int_X u(\omega+dd^c G_x) \wedge \omega^{n-1}=\frac{1}{V_{\omega}} \int_X G_x (\omega+dd^c u) \wedge \omega^{n-1}.
$$
In particular for $u=G_y$ we obtain the symmetry relation $G_x(y)=G_y(x)$.
Green's functions are classical objects of study in Riemannian geometry. In particular it is known 
that $G_x \in {\mathcal C}^{\infty}(X \setminus \{x \})$ and $G_x(y) \rightarrow -\infty$ as 
$y \rightarrow x$ (either at a logarithmic speed if $n=1$, or at a polynomial speed if $n \geq 2$).

\smallskip

While the Laplace operator $\Delta_{\omega}$ is linear, it depends on $\omega$ in a non linear way. 
Following \cite{GPS22,GPSS22,GPSS23} we are going to establish uniform estimates on normalized
$\omega$-sh functions, by comparing them with $\omega$-plurisubharmonic solutions to certain complex Monge-Amp\`ere equations.

  \subsubsection{$\omega$-plurisubharmonic functions}

  A function is quasi-plurisub\-harmonic (quasi-psh) if it is locally given as the sum of  a smooth and a psh function.   
 Quasi-psh functions
$\f:X \rightarrow \R \cup \{-\infty\}$ satisfying
$
\omega_{\f}:=\omega+dd^c \f \geq 0
$
in the weak sense of currents are called $\omega$-plurisubharmonic ($\omega$-psh for short).

\begin{defi}
We let $\PSH(X,\omega)$ denote the set of all $\omega$-plurisubharmonic functions which are not identically $-\infty$.  
\end{defi}

Note that constant functions are $\omega$-psh functions.
A ${\mathcal C}^2$-smooth function $u$ has bounded Hessian, hence $\e u$ is
$\omega$-psh if $0<\e$ is small enough and $\omega$ is  positive.

The set $\PSH(X,\omega)$ is a closed subset of $L^1(X)$, 
for the $L^1$-topology.
Subsets of  $\omega$-psh functions enjoy strong compactness and integrability properties,
we mention notably the following: for any fixed $r \geq 1$, 
\begin{itemize}
\item $\PSH(X,\omega) \subset L^r(X)$;
 the induced $L^r$-topologies are equivalent;
\item $\PSH_A(X,\omega):=\{ u \in\PSH(X,\omega), \, -A \leq \sup_X u \leq 0 \}$ is compact in  $L^r$.
\end{itemize}
We refer the reader to \cite{Dem12,GZbook} for further basic properties of $\omega$-psh functions.

   \subsubsection{Laplace vs Monge-Amp\`ere solutions} \label{sec:LaplacevsMA}
   
  We shall regularly compare  solutions to the Laplace equation 
   and $\omega$-psh solutions to an auxiliary Monge-Amp\`ere equation.
   The following principle will  play an important role.

\begin{prop} \label{pro:comparisonLMA}
Fix  $t>0$, $p>1$ and  $0 \leq f \in L^p(\omega^n)$. Let $v$ (resp. $\f$) be the unique bounded  $\omega$-sh (resp. $\omega$-psh)
function such that  
$$
 (\omega+dd^c v) \wedge \omega^{n-1}=e^{tv} f \omega^n
\; \; \text{ and } \; \; 
(\omega+dd^c \f)^n = e^{nt\f} f^n \omega^n.
$$
Then $\f \leq v$.
\end{prop}

The existence of $v$ is classical. 
For the existence and uniqueness of $\f$, see \cite{GZbook}.

\begin{proof}
It follows from the maximum principle that $v$ is the envelope of bounded $\omega$-sh subsolutions to this twisted Laplace equation, i.e. for all $x \in X$,
$$
v(x)=\sup \left\{ u(x); \; u \in SH(X,\omega) \cap L^{\infty}(X)
\; \text{ and } \; 
 (\omega+dd^c u) \wedge \omega^{n-1} \geq e^{tu} f \omega^n \right\}.
$$

Since the function $\f_2=0$ is the trivial solution to the twisted Monge-Amp\`ere equation
$(\omega+dd^c \f_2)^n = e^{nt\f_2} 1^n \omega^n$,
the AM-GM inequality (see \cite[Theorem 2.12]{DK14}) ensures that 
$$
(\omega+dd^c \f) \wedge \omega^{n-1} \geq e^{t \f}  f \omega^n.
$$
The conclusion follows as $PSH(X,\omega) \subset SH(X,\omega)$
and $\f$ is a bounded subsolution of the twisted Laplace equation.
\end{proof}

  Uniform a priori estimates for solutions to complex Monge-Amp\`ere equations
  have been intensively studied in the past decades. The previous proposition 
  will allow one to use them and gain uniform $L^{\infty}$ a priori estimates
  for solutions of the Laplace equation, as the reference form $\omega$ varies
  (see Section \ref{sec:uniformGreen}).

   \subsection{K\"ahler families} \label{sec:uniformMA}
   
   \subsubsection{Assumptions}
   
  We shall consider the following set of assumptions.
    
\begin{setting} \label{set}
Let $\mathcal{X}$ be an irreducible and reduced  complex space. We let $\pi:{\mathcal X} \rightarrow 2 \mathbb D$ 
denote a proper, surjective holomorphic map with connected fibers such that each fiber $X_t=\pi^{-1}(t)$ is a $n$-dimensional 
 compact K\"ahler manifold, for $t\in 2 \mathbb D^*$  and  $X_0$ is an irreducible and reduced complex space. 
We fix $\beta$ a relative K\"ahler form on $\mathcal{X}$, set $\beta_t=\beta_{|X_t}$  and assume that the volumes $V_{\beta_t}:=\int_{X_t} \beta_t^n$ are uniformly bounded away from 0 and $+ \infty $.  Let $dV_{X_t}=\beta_t^n /V_{\beta_t}$  denote the corresponding  probability volume form on $X_t$.  
\end{setting}

Here $\mathbb D$ is the unit disk in $\C$; we assume that the family is defined over a slightly larger disk
(e.g. $2 \mathbb D$), and will establish estimates that are uniform with respect to the parameter $t \in \mathbb D^*$.

\begin{defi}
Fix $p>1$ and $A,B>0$.
We let ${\mathcal K}(\mathcal{X},p,A,B)$ denote the set of all relative K\"ahler forms $\omega$  on 
$\mathcal{X}\setminus X_0$ such that  for  all $t \in \mathbb D^*$,  $\{\omega_t\}\leq A\{\beta_t\}$ in $H^{1,1}(X_t,\R)$, and
$$
\int_{X_t} f_t^p dV_{X_t} \leq B,
\; \; \text{ where } \; \; 
\frac{1}{V_{\omega_t}} \omega_t^n=f_t dV_{X_t}.
$$
\end{defi}

The cohomological assumption 
ensures that  the volumes $V_t=\int_{X_t} \omega_t^n \leq C(A)$ are uniformly bounded fom above,
but it is a strictly stronger assumption as we observe in the following example.

\begin{exa} \label{exa:voldiam}
Assume ${X}=\PP^1 \times \PP^1$ is the product of two Riemann spheres, endowed with the 
K\"ahler form $\omega_{\lambda}(x,y)=\lambda \omega_{\PP^1}(x)+\lambda^{-1} \omega_{\PP^1}(y)$,
where $\lambda>0$. Observe that the volume
$
V_{\omega_{\lambda}}=\int_X \omega_{\lambda}^2=\int_X 2 \omega_{\PP^1}(x) \wedge \omega_{\PP^1}(y)=2
$
is constant, while the diameter
${\rm diam}(X,\omega_{\lambda}) \sim \lambda \rightarrow \infty$ as $\lambda \rightarrow \infty$.
\end{exa}

    \subsubsection{Uniform a priori estimates}

We shall need the following uniform integrability result in families, that 
generalizes previous results of Skoda and Zeriahi:

\begin{thm} \label{thm:sz}
Fix $p>1$, $A,B>0$. There exists $\alpha=\alpha(n,p,A,B)>0$ such that for all 
$\omega \in {\mathcal K}(\mathcal{X},p,A,B)$, $t \in \D^*$ and
$\f_t \in PSH(X_t,\omega_t)$ with $\sup_{X_t} \f_t=0$,
$$
\int_{X_t} \exp(-\alpha \f_t) dV_{X_t} \leq C_{\alpha},
$$
where $C=C(\alpha,n,p,A,B)>0$ is independent of $\omega,t,\f_t$.
\end{thm}

\begin{proof}
The cohomological assumption means that there exists a smooth closed $(1,1)$-form $\theta_t$ on $X_t$,
 cohomologous to $\omega_t$, such that  $\theta_t \leq  A\beta_t$. 
 The $\partial\overline{\partial}$-lemma ensures that    $\omega_t=\theta_t+dd^c u_t$, where 
 $\int  _X u_t dV_{X_t}  =0$ and  $u_t \in PSH(X_t,\theta_t)$. 
 Since we also have $u_t \in PSH(X_t, A \beta_t)$,
 it follows from  \cite[Conjecture 3.1]{DnGG23} and  \cite[Corollary 4.8]{Ou22}
  that  $0 \leq \sup_{X_t} u_t \leq M_0$ for some uniform constant $M_0$.
  
  Set $\tilde{\f}_t=\f_t-\int_{X_t} \f_t dV_{X_t}$ so that $\int_{X_t} \tilde{\f}_t dV_{X_t}=0$.
 Observe similarly that  $ \psi_t=\tilde{\f}_t+u_t \in PSH(X_t, \theta_t) \subset PSH(X_t, A \beta_t)$ 
 is normalized by $\int_{X} \psi_t dV_{X_t}=0$,
 hence $0 \leq \sup_{X_t} \psi_t \leq M_0$.
It follows therefore from \cite[Theorem 2.9]{DnGG23} that  
$$
\int_{X_t} e^{2\alpha|\psi|}dV_{X_t} + \int_{X_t} e^{2\alpha| u |}dV_{X_t} \leq  C_\alpha,
$$
for some uniform constants  $\alpha,C_\alpha>0$.
Cauchy-Schwarz inequality now yields
  $$
\int_{X_t} e^{\alpha |\tilde{\f}_t| } dV_{X_t} 
\leq \int_{X_t} e^{\alpha |\psi_t| } e^{\alpha |u_t| } dV_{X_t} 
\leq  C_\alpha.
$$

Now \cite[Theorem 1.9]{DnGG23} ensures that $||u_t-V_{\theta_t}||_{\infty} \leq M_1$,
where $V_{\theta_t}=\sup \{v, \; v \in PSH(X,\theta_t) \; \text{ with } \; v \leq 0 \}$.
Thus $\p_t \leq V_{\theta_t}+\sup_{X_t} \p_t \leq u_t+M_0+M_1$ and $\sup_{X_t} \tilde{\f}_t \leq M_0+M_1$.
The conclusion follows as $-\f_t \leq |\tilde{\f}_t|+M_0+M_1$.
\end{proof}

  The following uniform estimate is the key tool for all results to follow.

   \begin{thm} \label{thm:uniformMA}
Fix $p>1$, $A,B>0$ and $\omega \in {\mathcal K}(\mathcal{X},p,A,B)$. 
Assume that there exists $\f_t \in PSH(X,\omega_t) \cap L^{\infty}(X_t)$,
$p'>1$ and $B'>0$ independent of $t$ such that
$$
\frac{1}{V_{\omega_t}}(\omega_t+dd^c \f_t)^n=g_t dV_{X_t},
$$
with $\int_{X_t} g_t^{p'} dV_{X_t} \leq B'$.
Then ${\rm Osc}_X(\f_t) \leq C=C(p,p',A,B,B')$.
\end{thm}

These uniform estimates have a long history. For a fixed cohomology class they have been established by Kolodziej
\cite{Kol98}, generalizing a celebrated a priori estimate of Yau \cite{Yau78}.
These have been further generalized in \cite{EGZ09,EGZ08,DP10,BEGZ10} in order to deal with less positive or collapsing families of cohomology classes on K\"ahler manifolds. Alternative proofs have been provided in \cite{GL21,GPT23},
while the family version needed here is obtained in \cite{DnGG23}.

\begin{proof}
The result follows from   \cite[Theorem A]{DnGG23} and Theorem \ref{thm:sz}.
\end{proof}

These uniform estimates remain valid under less restrictive integrability assumptions on the densities.
To gain clarity in the proofs to follow, we have chosen to restrict to this setting which is the most useful one 
in geometric applications. 
We refer the reader to \cite[Section 5]{GGZ23} for a discussion of the optimality of the
integrability assumptions that ensure finiteness of diameters of K\"ahler metrics.

\section{Green's functions} \label{sec:uniformGreen}

In this section we prove Theorem A. In the setting \ref{set} we fix
$p>1$, $A,B>0$ and $\omega \in {\mathcal K}(\mathcal{X},p,A,B)$.
 For $t \in {\mathbb D^*}$ we consider the Green function
 $G^{\omega_t}_{x}$ of the K\"ahler form $\omega_t$ at the point $x \in X_t$ and establish
integrability estimates for the latter  that are uniform in  $t \in {\mathbb D^*}$ and $x \in X_t$.
To lighten the notation, we get rid of the $t$-subscript and
write $X$ instead of $X_t$ and $G_x$ instead of $G^{\omega_t}_{x}$.

\subsection{Bounding $\omega$-subharmonic functions from above}

   \begin{lem}  \label{lem:5.1}
   Fix $a>0$ and let $v$ be a quasi-subharmonic function on $X$ such that $\Delta_{\omega} v \geq -a$ and  $\int_X v \omega^n=0$.
Then 
$$
\sup_X v \leq C \left[ a+ \frac{1}{V_{\omega}}  \int_X |v| \omega^n \right],
$$
where $C=C(n,p,A,B)>0$ only depends on $n,p,A,B$.
 \end{lem}
 
 This result is a family version of \cite[Lemma 5.1]{GPSS22}.

 \begin{proof}
Both the statement and the assumptions are homogeneous of degree $1$.
Changing $v$ in $\frac{n}{a} v$ we thus reduce to the case $a=n$.
Observe that $\Delta_{\omega} v \geq -n$ is equivalent to $(\omega+dd^c v) \wedge \omega^{n-1} \geq 0$.

Regularizing $v$ we can assume that it is smooth.
We set $v_+=\widetilde{\max}(v,0)$, where 
$\widetilde{\max}$ denotes a convex regularized maximum such that $0\leq \widetilde{\max} \leq 1+ \max$ .
 Since $\sup_X v \leq \sup_X v_+$,
 it suffices to  bound $v_+$ from above. 
Let $\f \in PSH(X,\omega)$ be the unique smooth function such that 
$$
(\omega+dd^c \f)^n=\frac{1+v_+}{1+M} \omega^n
$$
and $\sup_X \f=-1$, where  $M=\int_X v_+  \frac{\omega^n}{V_\omega} \leq 1+ \frac{1}{2} \int_X |v| \frac{\omega^n}{V_\omega}$.

\smallskip

Set $H=1+v_+-\e (-\f)^{\alpha}$, where $\alpha=\frac{n}{n+1}$ and 
$\e>0$ is chosen so that $\frac{\e^{n+1}\alpha^n}{(1+\alpha \e)^n}=1+M$.
We are going to show that $H \leq 0$.
Observe that
$$
-dd^c (-\f)^{\alpha}=\alpha(1-\alpha) (-\f)^{\alpha-2} d\f \wedge d^c \f +\alpha(-\f)^{\alpha-1} dd^c \f,
$$
 hence $\Delta_{\omega}(-\e (-\f)^{\alpha}) \geq \alpha \e(-\f)^{\alpha-1} \Delta_{\omega} \f
 \geq n \alpha \e(-\f)^{\alpha-1} \left[ \left(\frac{1+v_+}{1+M} \right)^{\frac{1}{n}} -1 \right]$.
 We infer
 $$
 \Delta_{\omega}H \geq -n+n \alpha \e(-\f)^{\alpha-1} \left[ \left(\frac{1+v_+}{1+M} \right)^{\frac{1}{n}} -1 \right].
 $$
At the point $x_0$ where $H$ reaches its maximum, we have $0 \geq  \Delta_{\omega}H $ hence
$$
(1+\alpha \e)(-\f)^{1-\alpha} \geq (-\f)^{1-\alpha}+\alpha \e \geq 
\alpha \e \left(\frac{1+v_+}{1+M} \right)^{\frac{1}{n}},
$$
using that $(-\f)^{1-\alpha} \geq 1$. Thus
$$
\e (-\f)^{\alpha}=\e (-\f)^{n(1-\alpha)} \geq \frac{\alpha^n \e^{n+1}}{(1+\alpha \e)^n} \frac{1+v_+}{1+M}=1+v_+,
$$
by our choice of $\e$ and $\alpha$. This shows that $H \leq 0$
hence $1+v_+ \leq \e (-\f)^{\alpha}$.

\smallskip

Note that $\e \leq c_n (1+M)$ since $\frac{\e^{n+1}\alpha^n}{(1+\alpha \e)^n}=1+M$.
Thus   $\frac{(\omega+dd^c \f)^n}{V_{\omega}}=F  dV_X$
with
$$
F=\frac{1+v_+}{1+M} f \leq \frac{\e (-\f)^{\alpha}}{1+M} f \leq c_n (-\f)^{\alpha} f.
$$
Since   $\int_X f^p dV_X \leq A$,
we can fix $1<r<p$ and use
H\"older inequality to obtain
$$
\int_X F^r dV_X \leq \e^r \int_X (-\f)^{r\alpha} f^r dV_X \leq 
\left( \int_X f^p dV_X \right)^\frac{r}{p}
\left( \int_X (-\f)^{\frac{rp\alpha}{p-r}} dV_X \right)^\frac{p-r}{p}.
$$
The first integral is controlled by $A$ by assumption, the second one is uniformly bounded
by Theorem \ref{thm:sz}.
Thus  $\f$ is uniformly bounded by Theorem \ref{thm:uniformMA}, hence 
$$
\sup_X v \leq \sup_X v_+ \leq \e  (-\f)^{\alpha} \leq c_n [1+M] C_0,
$$
and the conclusion follows as $M\leq \frac{1}{2V_\omega} \int_X |v| \omega^n+1$.
 \end{proof}

 The following result is a key improvement on the results obtained in \cite{GPSS22}.

    \begin{prop}  \label{prop:key}
  Let $u$ be a  continuous function such that $\int_X u \omega^n=0$
  and $|\Delta_{\omega} u| \leq 1$. Then 
$$
||u||_{L^{\infty}(X)} \leq C,
$$
where $C=C(n,p,A,B)>0$ only depends on $n,p,A,B$.
 \end{prop}

 \begin{proof} 
 We let $C_1>0$ denote the uniform constant from Lemma \ref{lem:5.1} and we set
 $$
\delta=\frac{1}{4(1+4n^2C_1^2)^{2}}.
$$
 
 Observe that the statement to be proved is homogeneous of degree $1$, so it suffices to show that
if $u$ is a  continuous function such that $\int_X u \omega^n=0$ and $|\Delta_{\omega} u| \leq \delta$,
then $M=\frac{1}{V_\omega}\int_ X |u|\omega^n \leq C$ is uniformly bounded from above independently of $u$.
We   assume that $M \geq 1$ (otherwise we are done), and we set
$$
v:=\frac{u}{M} =\varepsilon u
\; \; \text{ where } \; \;
0<\varepsilon:= \frac{1}{M}\leq 1.
$$
Since $\|\Delta_\omega v\|_{L^\infty(X)} \leq  1, \frac{1}{V_\omega} \int_X |v|\omega^n=1$ 
and $\int_X v\omega^n=0$, Lemma \ref{lem:5.1} yields
\begin{equation}\label{eq_bound_v}
\|v\|_{L^\infty(X)} \leq 2C_1.
\end{equation} 

Set
 $$
 H:=\frac{1}{n}\Delta_\omega u= \frac{dd^c u\wedge \omega^{n-1}}{\omega^n}
 \; \; \text{ so that } \; \;
 \|H\|_{L^\infty(X)}\leq  \delta.
 $$ 

Set $t=\sqrt{\delta} \in (0,1)$ and observe that the function $v$ is $\varepsilon\omega$-sh
and satisfies
$$
(\varepsilon\omega+dd^c v)\wedge (\varepsilon\omega)^{n-1}=\varepsilon^n(1+H)\omega^n= e^{tv }e^{-tv}(1+H) (\varepsilon \omega)^n .
$$

We let $\varphi\in PSH(X, \varepsilon \omega)\cap L^{\infty}(X)$ be the unique bounded
$\e\omega$-psh solution of the complex Monge-Amp\`ere equation
$$
(\varepsilon \omega +dd^c \varphi )^n= e^{nt\varphi }e^{-ntv}(1+H)^n (\varepsilon \omega)^n.
$$
It follows from Proposition \ref{pro:comparisonLMA} 
(applied to $\e \omega$ and $f=e^{-tv}(1+H)$) that 
\begin{equation}\label{eq_phi_v}
\varphi\leq v.
\end{equation}

Setting $ \psi=\varphi/ \varepsilon \in PSH(X, \omega)$, the equation can be rewritten as
\begin{equation}\label{eq_psi}
 ( \omega +dd^c \psi )^n= e^{nt(\varepsilon\psi-v )}(1+H)^n  \omega^n \leq 2^n \omega^n,
\end{equation}
since $\varepsilon\psi-v=\varphi-v \leq 0$ and $\|H\|_{L^\infty}\leq  \delta \leq 1$. 
 It follows from Theorem \ref{thm:uniformMA} that 
$$
{\rm Osc}_X (\psi)= \|\tilde{\psi}\|_{L^\infty(X)}\leq C_0=C_0(n,p,A,B),
$$
 where $\tilde{\psi}:=\psi-\sup_X \psi \leq 0$.
Integrating  \eqref{eq_psi} we obtain
$$
1= e^{nt\varepsilon \sup_X \psi} \int_X e^{nt\varepsilon\tilde{\psi}}e^{-ntv}(1+H)^n\frac{\omega^n}{V_\omega}
\leq e^{nt\varepsilon \sup_X \psi}(1+\delta)^n\int_X e^{-ntv}\frac{\omega^n}{V_\omega}.
$$
We let the reader check that $e^x \leq 1+x+x^2$ for $|x| \leq 1$. 
Since $t<n^{-1}C_1^{-1}$ we infer
\begin{align*}
\int_X e^{-ntv}\frac{\omega^n}{V_\omega}&\leq 1- nt\int_X v\frac{\omega^n}{V_\omega} + n^2 t^2\int_X v^2\frac{\omega^n}{V_\omega}\\
&= 1+ n^2 t^2\int_X v^2\frac{\omega^n}{V_\omega} \quad (\text{ since } \int_X v\omega^n=0) \\
&\leq   1+  4n^2 C_1^2 t^2,
\end{align*}
since $||v||_{L^\infty(X)} \leq 2C_1$.
Using that $\log(1+x) \leq x$ we therefore obtain
\begin{align*}
nt\varepsilon \sup_X \psi &\geq - n\log(1+\delta)- \log \left(1+4n^2C_1^2 t^2 \right)\\
&\geq - n \delta -4n^2 C_1^2 t^2.
\end{align*}
Using that  $t= \sqrt{\delta}=\frac{1}{2[1+4n^2C_1^2]}$ we conclude that
\[
\varepsilon \sup_X \psi \geq -\frac{1}{2}.
\]

It follows that 
$ \varphi= (\varphi -\sup_X \varphi)+  \sup_X \varphi = \varepsilon \tilde{\psi}+ \varepsilon \sup_X \psi 
\geq  -C_0 \varepsilon  -\frac{1}{2}$.
Together with \eqref{eq_phi_v}, this shows that
$$
v\geq - C_0 \varepsilon -  \frac{1}{2}.
$$
Repeating the same argument for $-v$, we get $-v\geq - C_0\varepsilon - \frac{1}{2}$ hence
\begin{align*}
\|v\|_{L^\infty(X)}\leq  C_0 \varepsilon+ \frac{1}{2}.
\end{align*}
Therefore
$$
1=\frac{1}{V_\omega}\int_X |v|\omega^n\leq \frac{1}{V_\omega}\int_X \|v\|_{L^\infty(X)}\omega^n \leq C_0 \varepsilon+ \frac{1}{2},
$$
which yields
$$
\frac{1}{V_\omega} \int_X |u|\omega^n \leq 2 C_0.
$$

Unravelling the normalizations we have made, we see that the constant $C$ from the statement can be chosen as
$C=8C_0[1+4n^2C_1^2]^2$, where $C_0$ is the uniform constant provided by Theorem \ref{thm:uniformMA} 
(taking $g_t=2^n \omega_t^n/dV_{X_t}$)
and  $C_1$ is the uniform constant from Lemma \ref{lem:5.1}.
 \end{proof}

 \subsection{Green's functions}  \label{sec:green}

   \begin{theorem}  \label{thm:Green1}
   Fix $0<r<\frac{n}{n-1}$ and $0<s< \frac{2n}{2n-1}$.
For all $t \in \D^*, x \in X_t$ and  $\omega \in {\mathcal K}(\mathcal{X},p,A,B)$ we have
\begin{enumerate}
\item $\sup_{X_t} G^{\omega_t}_x \leq C_0=C_0(n,p,A,B)$;

\smallskip

\item $\frac{1}{V_{\omega_t}} \int_{X_t} |G^{\omega_t}_x|^r  \omega_t^n  \leq C_1=C_1(n,p,r,A,B)$;

\smallskip

\item $\frac{1}{V_{\omega_t}} \int_{X_t} |\nabla G^{\omega_t}_x|_{\omega_t}^s \omega_t^n \leq C_2=C_2(n,p,s,A,B)$.
\end{enumerate}
 \end{theorem}
 
  Given    Proposition \ref{prop:key} above, the proof is a combination of 
  the main results  of \cite{GPS22,GPSS22} 
  with the uniform estimates provided by Theorem \ref{thm:uniformMA}.

 \begin{proof}
 {\it Step 1}.
  Consider $h=-{\bf 1}_{\{G_x \leq 0 \} }+\int_{\{G_x \leq 0 \}} \frac{\omega^n}{V_{\omega}}$.
 Observe that  $-1 \leq h \leq 1$ and $\int_X h \omega^n=0$. We let $v$ denote the unique
 solution $\Delta_{\omega} v=h$ with $\int_X v \omega^n=0$.
 It follows from Proposition \ref{prop:key} that $||v||_{L^{\infty}(X)} \leq C$.  
 Thus
 \begin{eqnarray*}
C \geq   v(x) &=& \frac{1}{V_{\omega}} \int_X v (\omega+dd^c G_x) \wedge \omega^{n-1} \\
& =& \frac{1}{V_{\omega}} \int_X G_x dd^c v \wedge \omega^{n-1}
 =n \int_{\{G_x \leq 0 \}} (-G_x) \frac{\omega^n}{V_{\omega}}.
 \end{eqnarray*}
Since $\int_X G_x \omega^n=0$, we infer
$$
\int_X |G_x| \frac{\omega^n}{V_{\omega}} =
2 \int_{\{G_x \leq 0 \}} (-G_x) \frac{\omega^n}{V_{\omega}} \leq \frac{2C}{n}.
$$
 It therefore follows from Lemma \ref{lem:5.1} that $\sup_X G_x \leq C_0$, proving (1).
 
 \medskip
 
\noindent   {\it Step 2}.
We   have shown (2) for $r=1$ in the previous step.
 We now show (2) for $r<1+\frac{1}{n}$
Set ${\mathcal G}_x=G_x-C_0-1 \leq -1$ and
consider $u$   the   $\omega$-sh solution of
$$
\frac{1}{V_{\omega}} (\omega+dd^c u) \wedge \omega^{n-1} =
\frac{(-{\mathcal G}_x)^{\beta} \omega^n}{\int_X (-{\mathcal G}_x)^{\beta} \omega^n},
$$
with $\int_X u \omega^n=0$, where $0<\beta < \frac{1}{n}$. We are going to show that 
$u \geq -C$ is uniformly bounded below. It will follow that
\begin{eqnarray*}
-C \leq u(x) &=& \int_X u \frac{(\omega+dd^c {\mathcal G}_x) \wedge \omega^{n-1}}{V_{\omega}} \\
&=& \int_X {\mathcal G}_x \frac{(\omega+dd^c u) \wedge \omega^{n-1}}{V_{\omega}} 
= -\frac{\int_X (-{\mathcal G}_x)^{1+\beta}  \frac{\omega^n}{V_{\omega}}}{\int_X (-{\mathcal G}_x)^{\beta}  \frac{\omega^n}{V_{\omega}}}.
\end{eqnarray*}
Since $1 \leq -{\mathcal G}_x$  we obtain 
$\int_X (-{\mathcal G}_x)^{\beta}  \frac{\omega^n}{V_{\omega}} \leq \int_X (-{\mathcal G}_x)  \frac{\omega^n}{V_{\omega}}=1+C_0$, hence
$$
\int_X (-{\mathcal G}_x)^{1+\beta}  \frac{\omega^n}{V_{\omega}} \leq C[1+C_0],
$$
proving (2) for $r=1+\beta$, as ${\mathcal G}_x$ differs from $G_x$ by a uniform additive constant.

\smallskip

To prove that $u$ is uniformly bounded below, we
 consider  the normalized solution $\f \in PSH(X,\omega) \cap L^{\infty}(X)$, $\sup_X \f=0$,  
of the Monge-Amp\`ere equation
$$
\frac{1}{V_{\omega}} (\omega+dd^c \f)^n = \frac{(-{\mathcal G}_x)^{n\beta} \omega^n}{\int_X (-{\mathcal G}_x)^{n\beta} \omega^n}.
$$
Since $-{\mathcal G}_x \geq 1$, the density of the RHS is bounded from above by $(-{\mathcal G}_x)^{n\beta} f_{\omega}$.
It follows from H\"older inequality that the latter belongs to $L^{p'}(dV_X)$, $p'<p$, since
\begin{eqnarray*}
\int_X (-{\mathcal G}_x)^{n\beta p'} f_{\omega}^{p'} dV_X
&=& \int_X (-{\mathcal G}_x)^{n\beta p'} f_{\omega}^{p'-1} \frac{\omega^n}{V_{\omega}} \\
&\leq &  \left( \int_X f_{\omega}^p dV_X \right)^{\frac{p'-1}{p-1}} 
\left( \int_X (-{\mathcal G}_x)^{n\beta p's'} \frac{\omega^n}{V_{\omega}} \right)^{\frac{1}{s'}}
\leq C',
\end{eqnarray*}
where $s'=\frac{p-1}{p-p'}$ is the conjugate exponent of $s=\frac{p-1}{p'-1}$:
we choose $p'>1$ very close to $1$ (thus $s'>1$ is very close to $1$ as well) so that
$n\beta p' s'<1$, and the last integral is under control by the first step.
Theorem \ref{thm:uniformMA} now ensures that 
$$
-M_0 \leq \f \leq 0
$$
for some uniform   $M_0$.
Since $\int_X (-{\mathcal G}_x)^{n\beta} \frac{\omega^n}{V_{\omega}} 
\leq \int_X (-{\mathcal G}_x) \frac{\omega^n}{V_{\omega}} \leq 1+C_0=:C_1^n$,
it follows from the AM-GM inequality that $(\omega+dd^c \f) \wedge \omega^{n-1} \geq \frac{(-{\mathcal G}_x)^{\beta}}{C_1} \omega^n$,
while $ (-{\mathcal G}_x)^{\beta}  \omega^n \geq (\omega+dd^c u) \wedge \omega^{n-1}$ since $-{\mathcal G}_x \geq 1$.
We infer that $\f-u/C_1$ is a normalized $\omega$-sh function;
it is   bounded from above by Lemma \ref{lem:5.1}, hence
$$
-M_0C_1-C_2 \leq C_1 \f -C_2 \leq u \leq C_2',
$$
showing that $u$ is uniformly bounded, as claimed.

\medskip

\noindent   {\it Step 3}.
We now  establish (2) for the optimal values of $r$ by a recursive argument.
Indeed the reasoning from {\it Step 2} shows that if $\int_X (-{\mathcal G}_x)^{n\beta} \frac{\omega^n}{V_{\omega}} \leq C$
for $n\beta \leq \beta'$, then $\int_X (-{\mathcal G}_x)^{r} \frac{\omega^n}{V_{\omega}} \leq C'$
for $r<1+\frac{\beta'}{n}$. By induction this yields, for all $k \in \N$,
$$
\int_X (-{\mathcal G}_x)^{r} \frac{\omega^n}{V_{\omega}} \leq C_r
\; \; \text{ for }
r<1+\frac{1}{n}+\frac{1}{n^2}+\cdots+\frac{1}{n^k}.
$$
Thus a uniform control can be obtained for all $r<\frac{n}{n-1}$.

\medskip

\noindent   {\it Step 4}.
It follows from {\it Step 3}, Lemma \ref{lem:weightedgradient} below and H\"older inequality
that (3) holds for $s<\frac{2n}{2n-1}$.
Indeed set $r=\frac{s}{2-s}(1+\beta)$, and observe that 
by choosing $0<\beta$ very small and $r$ arbitrarily close to $\frac{n}{n-1}$,
we obtain $s$ arbitrarily close to  $\frac{2n}{2n-1}$.
Setting $2\alpha=s(1+\beta)$, we thus get
\begin{eqnarray*}
\int_X |\nabla {G}_x|^s  {\omega^n}= 
\int_X |\nabla {\mathcal G}_x|^s  {\omega^n} 
&=& \int_X \frac{|\nabla {\mathcal G}_x|^s}{|{\mathcal G}_x|^{\alpha}} |{\mathcal G}_x|^{\alpha} {\omega^n}  \\
&\leq & \left( \int_X \frac{|\nabla {\mathcal G}_x|^2}{{|{\mathcal G}_x}|^{\frac{2\alpha}{s}}}  {\omega^n}  \right)^{\frac{s}{2}}
\left( \int_X |{\mathcal G}_x|^{\frac{2\alpha}{2-s}}  {\omega^n}  \right)^{\frac{2-s}{2}} \\
&=& \left( \int_X \frac{|\nabla {\mathcal G}_x|^2}{{|{\mathcal G}_x}|^{1+\beta}}  {\omega^n}  \right)^{\frac{s}{2}}
\left( \int_X |{\mathcal G}_x|^{r}  {\omega^n}  \right)^{\frac{2-s}{2}} \leq C(s).
\end{eqnarray*}
 \end{proof}

 \begin{rem}
 All these estimates are valid, more generally, for any $\omega$-psh function $v$ which is 
 normalized by $\int_X v \omega^n=0$. Indeed using Stokes theorem we obtain
 $$
 v(x) =\frac{1}{V_{\omega}} \int_X G_x (\omega+dd^c v) \wedge \omega^{n-1} \leq \sup_X G_x \leq C_0.
 $$
 The other estimates  follow by using the symetry $G_x(y)=G_y(x)$.
 \end{rem}

 We have used the following observation \cite[Lemma 5.6]{GPSS22}.
 
 \begin{lem} \label{lem:weightedgradient}
 Fix $\beta>0$. Then
 $$
 \frac{1}{V_{\omega}} \int_X \frac{d G_x \wedge d^c G_x \wedge \omega^{n-1}}{(-G_x+C_0+1)^{1+\beta}} \leq \frac{1}{\beta}.
 $$
 \end{lem}
 
 Here $C_0$ denotes the uniform constant from Theorem \ref{thm:Green1}.1.
 
 \begin{proof}
 The function $u(y)=(-G_x(y)+C_0+1)^{-\beta}$ takes values in $[0,1]$ with $u(x)=0$. 
 Since $du=\frac{\beta dG_x}{(-G_x+C_0+1)^{1+\beta}}$, we infer
 $$
 0=\frac{1}{V_{\omega}} \int_X u (\omega+dd^c G_x) \wedge \omega^{n-1}
 =\frac{1}{V_{\omega}} \int_X u \omega^{n}
 - \frac{\beta}{V_{\omega}} \int_X \frac{d G_x \wedge d^c G_x \wedge \omega^{n-1}}{(-G_x+C_0+1)^{\beta+1}}.
 $$
 The result follows.
 \end{proof}

 \subsection{Sobolev estimates}  \label{sec:sobolev}

The following is our improved and family version of \cite[Theorem 2.1 and Lemma 6.2]{GPSS23}.

   \begin{theorem}  \label{thm:sobolev}
   Fix $1<r< \frac{2n}{n-1}$, $t \in \D^*$ and $\omega \in {\mathcal K}(\mathcal{X},p,A,B)$.
   
   1) For all $u\in W^{1,2}(X_t)$, we have 
$$
 \left(\frac{1}{V_{\omega_t}} \int_{X_t} |u-\overline{u}|^{2r}\omega_t^n\right)^{1/r}
 \leq C_1\frac{1}{V_{\omega_t}} \int_{X_t} |\nabla u|^2_{\omega_t} \omega_t^n,
$$
where $\overline u= \frac{1}{V_{\omega_t}}\int_{X_t} u \omega_t^n$ and $C_1=C_1(n,p,r,A,B)>0$.
    
    2) If $\Omega \subset X_t$ is a domain and $u\in W^{1,2}(\Omega)$
    has compact support in $\Omega$, then 
    $$
 \left( \frac{1}{V_{\omega_t}} \int_{\Omega} |u|^{2r}\omega_t^n \right)^{1/r}
 \leq C_2 \left[1+ \frac{V_{\omega_t}(\Omega)}{V_{\omega_t}(X_t \setminus \Omega)} \right] 
 \frac{1}{V_{\omega_t}}  \int_{\Omega} |\nabla u|^2_{\omega_t} \omega_t^n,
$$
where $C_2=C_2(n,p,r,A,B)>0$.
 \end{theorem}

 \begin{proof}
The proof is very similar to that of \cite[Theorem 2.1 and Lemma 6.2]{GPSS23}, 
so we only sketch it. We fix  $ \beta\in (0,1)$ such that  $(1+\beta)r< \frac{n}{n-1}$.  
Green's formula and  H\"older inequality yield
\begin{align*}
|u(x)-\bar u| &= \left| \frac{1}{V_\omega}\int_X  du\wedge d^c {\mathcal{G}_x}\wedge \omega^{n-1} \right|\\
&\leq   \left( \frac{1}{V_\omega}  \int_X \frac{d \mathcal{G}_x \wedge d^c \mathcal{G}_x\wedge \omega^{n-1}}{(-\mathcal{G}_x)^{1+\beta}}   \right)^{1/2} \left(  \frac{1}{V_\omega}\int_X ( - \mathcal{G}_x)^{1+\beta} |\nabla u|^2_\omega \omega^n \right)^{1/2}\\
&\leq \frac{1}{\beta^{1/2}} \left(  \frac{1}{V_\omega}\int_X ( - \mathcal{G}_x)^{1+\beta} |\nabla u|^2_\omega \omega^n \right)^{1/2},
\end{align*}
hence
\begin{equation}\label{eq:lhs}
 \left( \int_X |u-\overline{u}|^{2r}\omega^n\right)^{1/r}
 \leq\frac{1}{\beta} \Vert \frac{1}{V_\omega}\int_X ( - \mathcal{G}_x)^{1+\beta} |\nabla u|^2_\omega \omega^n \Vert_{L^r(X, \omega)}. 
\end{equation}
It follows from Minskowski's inequality for integrals that
\begin{align*}
\Vert \frac{1}{V_\omega}\int_X ( - \mathcal{G}_x)^{1+\beta} |\nabla u|^2_\omega \omega^n \Vert_{L^r(X, \omega)} &\leq   \frac{1}{V_\omega} \int_X \left(  \int_X  ( - \mathcal{G}_x)^{r(1+\beta)}\omega^n(x) \right)^{\frac{1}{r}}  |\nabla u|^2_\omega(y) \omega^n(y) \\
&\leq C_1   \frac{V_{\omega}^{1/r}}{V_{\omega}} \int_X |\nabla u|^2_\omega \omega^n, 
\end{align*}
using Theorem \ref{thm:Green1}.2. 
Together with \eqref{eq:lhs} we obtain
$$
 \left(\frac{1}{V_{\omega}} \int_X |u-\overline{u}|^{2r}\omega^n\right)^{1/r}
 \leq C_1\frac{1}{V_{\omega}} \int_X |\nabla u|^2_\omega \omega^n.
$$

\medskip

For the second inequality, as  in \cite[Lemma 6.2]{GPSS23}, using Green formula, one has for any $x\in \Omega$ and $\beta>0$,

\[
|u(x)|^2\leq \frac{C}{V_\omega(\Omega^c)}\int_X |\nabla u|^2_\omega \omega^2 + \frac{1}{ \beta V_\omega}  \int_X  (-\mathcal{G}(x,y))^{1+\beta}|\nabla u|^2_\omega \omega^n.
\]
Raising to the power $r$ on both sides   and arguing similarly to what we have done above, we obtain the required inequality.
 \end{proof}

  \section{Geometric applications}

  \subsection{Diameter bounds and non-collapsing}
  
  Among the various applications of Theorem A, we stress the following  diameter and non-collapsing estimates.
 
   \begin{thm}  \label{thm:diameter}
  Fix $0<\delta<1$ and $\omega \in {\mathcal K}(\mathcal{X},p,A,B)$. 
  There exists $C=C(n,p,\delta,A,B)>0$ such that  for all $ t \in \D^*, x \in X_t, $ and $r>0$,
  $$
  C  \min(1,r^{2n+\delta}) \leq \frac{1}{V_{\omega_t}} {\rm Vol}_{\omega_t}(B_{\omega_t}(x,r))
  \; \; \text{ and } \; \; 
  {\rm diam}(X_t,\omega_t) \leq C.
  $$
 Thus the family of compact metric spaces 
  $\{ (X_t,d_{\omega_t}), \; \omega \in {\mathcal K}(\mathcal{X},p,A,B), \,  t \in \D^* \}$
  is pre-compact in the Gromov-Hausdorff topology.
 \end{thm}
 
Due to its prominent role in geometric analysis, there has been an intensive search for such uniform diameter estimates in the past decade. We simply list the most recent contributions which require as an extra assumption 
\begin{itemize}
\item either a Ricci lower bound \cite{LTZ17,FGS20,GS22};
\item or  that $X$ be of general type \cite{Bam18,Wang18,JS22};
\item to control the  continuity of Monge-Amp\`ere potentials \cite{Li21,GGZ23,V23};
\item orelse a uniform lower bound on  $\omega^n/dV_X$ \cite{GPS22,GPSS22,GPSS23}.
\end{itemize}

The proof of Theorem \ref{thm:diameter} is similar to that of \cite[Theorem 1.1]{GPSS22}.

 \begin{proof}   
We fix $(x_0,y_0)\in X^2$ such that $d_\omega(x_0,y_0)={\rm diam}(X, \omega)$.
 The function $\rho: x \in  X \mapsto  d_\omega(x_0, x) \in \mathbb{R}^+$ is $1$-Lipschitz 
  with $\rho(x_0)=0$.   Green's formula applied to the function $\rho$ at the point $x_0$ yields
\[
\int_X \rho \omega^n  =\int_X d \rho  \wedge d^c G_{x_0}  \wedge\omega^{n-1}  \leq \int_X |\nabla G_{x_0}|_{\omega}\omega^n .
\] 
Using again Green's formula  at $y_0$ and  the previous inequality, we obtain
\begin{eqnarray*}
{\rm diam(X, \omega)} 
&=& \frac{1}{V_\omega} \int_X \rho\omega^n  -  \frac{1}{V_\omega} \int_X d \rho  \wedge d^c G_{y_0}  \wedge\omega^{n-1}\\
&\leq &  \frac{1}{V_\omega}\int_X |\nabla G_{x_0}|_{\omega}\omega^n
+\frac{1}{V_\omega}\int_X |\nabla G_{y_0}|_{\omega}\omega^n \leq C, 
\end{eqnarray*}
where the last inequality follows from Theorem \ref{thm:Green1}.3.

\smallskip

Next we prove the non-collapsing estimate.  
We fix $x \in X$ and consider  the uniformly bounded function $\rho: y \in X \mapsto  d_\omega(x, y) \in \mathbb{R}^+$.
Fix $r\in (0,1]$ and let  $\chi$ be a non-negative smooth cut-off function with support in $B_\omega(x, r)$ such that  
$\chi\equiv 1$ on  $\overline{B}_\omega(x, \frac{r}{2}) $ and $\sup_X |\nabla \chi|_\omega \leq \frac{C}{r}$.   
Thus $\rho \chi$ is a $C$-Lipschitz function.

We pick $s\in (1, \frac{2n}{2n-1})$ and denote by  $s^*=\frac{s}{s-1} \in (2n, \infty)$ the conjugate exponent of $s$. 
Applying Green's formula to  $\rho \chi$ at $y\in  \overline{B_{\omega}(x,r )}^c $,  we obtain
\begin{align*}
 \int_X \rho\chi \omega^n  &=  \frac{1}{V_\omega}  \int_X  d(\chi \rho) \wedge d^c G_y \wedge \omega^{n-1} \\
 &\leq  C 
 \left(\int_X |\nabla G_y|^s_\omega \omega^n\right)^{1/s} \left({\rm Vol}_\omega(B_{\omega}(x,r)) \right)^{1/s^*}\\
 &\leq CV_\omega^{1/s} \left({\rm Vol}_\omega(B_{\omega}(x,r)) \right)^{1/s^*},
\end{align*}
by Theorem \ref{thm:Green1}.3.
Applying Green's formula again with $z\in \partial B_\omega(x, r/2)$, we infer  
\begin{align*}
\frac{r}{2} &=\frac{1}{V_\omega} \int_X \rho\chi \omega^n 
+ \frac{1}{V_\omega}  \int_X  d(\chi \rho) \wedge d^c G_z \wedge \omega^{n-1} \\
&\leq 2CV_\omega^{-1+1/s} \left({\rm Vol}_\omega(B(x,r)) \right)^{1/s^*}= 2C\left(  \frac{{\rm Vol}_\omega(B(x,r)) }{V_\omega} \right)^{1/s^*}.
\end{align*}
Since $s^*=\frac{s}{s-1}\in (2n, \infty)$, this implies the non-collapsing estimate.

\medskip
Set ${\mathcal F}:=\{ (X_t,d_{\omega_t}), \; \omega \in {\mathcal K}(\mathcal{X},p,A,B), \,  t \in \D^* \}$.
Gromov's theorem \cite[Theorem 7.4.15]{BBI01} ensures that this family is pre-compact in the Gromov-Hausdorff
topology iff there is a uniform bound on the diameters and for each $\e>0$
one can find in each $X \in {\mathcal F}$ an $\e$-net consisting of no more than 
$N=N(\e)$-points. This follows easily from the uniform non-collapsing estimate, together
with the uniform upper bound on the global volumes ${\rm Vol}_{\omega_t}(X_t) \leq V_0$.
 \end{proof}

 \subsection{Diameters of smoothable cscK metrics} \label{sec:csck}

In this section we consider $(X,\beta)$ a compact $n$-dimensional K\"ahler variety with Kawamata log terminal (klt) singularities
which admits a $\Q$-Gorenstein smoothing
$\pi: {\mathcal X} \rightarrow \D$, i.e. 
\begin{itemize}
\item ${\mathcal X}$ is a $\Q$-Gorenstein complex space of complex dimension $n+1$,
\item $\pi$ is a proper surjective holomorphic map such that ${\mathcal X}_{|\pi^{-1}(0)} \sim X_0$,
 \item $X_t={\mathcal X}_{|\pi^{-1}(t)}$ is smooth for all $t \in \D^*$,
 \item there is a smooth    form $\beta_{{\mathcal X}}$ such that 
 $\beta_t={\beta_{{\mathcal X}}}_{|X_t}$ is K\"ahler with $\beta_0=\beta$.
\end{itemize}

When the Mabuchi functional $M_{\beta}$ is coercive, it has been shown in \cite[Theorem C]{PTT23} that
so are the Mabuchi functionals $M_{\beta_t}$ , hence
there exists a unique constant scalar curvature K\"ahler metric $\omega_t$ cohomologous to $\beta_t$.

    \begin{coro} \label{cor:csck}
   	In the setting above, there exists   $D>0$ such that for all $t \in \D^*$,
   	$$
   	{\rm diam}(X_t,\omega_t) \leq D.
   	$$
   \end{coro}

   \begin{proof}
For $t\in \mathbb{D}^* $, the unique constant scalar curvature K\"ahler metric $\omega_t\in [\beta_t]$ 
satisfies the following coupled equations:
$$
\left\{\begin{matrix}
   (\beta_t+ dd^c  \varphi_t)^n = e^{F_t} \beta_t^n  \\ 
  \Delta_{\omega_t} F_t= -\bar s_t + {\rm Tr}_{\omega_t} \Ric(\beta_t).
\end{matrix}\right. 
 $$
 
 The volumes $\int_X \beta_t^n$ are uniformly bounded away from zero and infinity, while
it follows from  \cite[Theorem 5.3]{PTT23} that 
the  smooth densities $f_t=e^{F_t}$ 
 satisfy $\|f_t\|_{L^p(X_t, \beta_t^n)} \leq B$,  for some $p>1$  and  for all $t\in \mathbb{D}^*$.  
Thus $\omega\in {\mathcal K}(\mathcal{X},p,1,B)$, and the uniform bound for ${\rm diam }(X_t, \omega_t) $ follows from  Theorem \ref{thm:diameter}. 
   \end{proof}

   \subsection{Estimates along the K\"ahler-Ricci flow} \label{sec:krf}
   
   We assume here that $X$ is a compact K\"ahler manifold with $K_X$ nef (a  smooth minimal model).
   We consider
   $$
   \frac{\partial \omega_t}{\partial t}=-{\rm Ric}(\omega_t)-\omega_t
   $$
   the normalized K\"ahler-Ricci flow starting from an initial K\"ahler metric $\omega_0$.
   It follows from \cite{TZ06} that this   flow exists for all times $t>0$.
   We refer the reader to \cite{Tos18,Tos24} for recent overviews of the theory of the K\"ahler-Ricci flow.
   
   It has been
   a challenging open problem up to now to obtain uniform geometric bounds along the flow  as $t \rightarrow +\infty$.
 We obtain here the following striking result which -in particular- solves \cite[Conjecture 6.2]{Tos18}:
    
      \begin{theorem} \label{thm:fkr}
      Fix $0<\delta<1$.
   There exist $C=C(\omega_0)>0$ and $c=c(\omega_0,\delta)>0$ such that for all $t>0$ and $x \in X$,
   $$
   {\rm diam}(X,\omega_t) \leq C  
   \; \; \text{ and } \; \;
     {\rm Vol}_{\omega_t}(B_{\omega_t}(x,r)) \geq c  r^{2n+\delta} V_{\omega_t},
   $$
   whenever $0<r<{\rm diam}(X,\omega_t)$.
   \end{theorem}
   
  This result has been established by Guo-Phong-Song-Sturm under the extra assumption that 
  the Kodaira dimension ${\rm kod}(X)$ of $X$ is non-negative (see \cite[Theorem 2.3]{GPSS22}).
  Recall that 
  $$
  {\rm kod}(X)=\limsup_{m \rightarrow+\infty} \left[ \frac{\log \dim H^0(X,mK_X)}{\log m} \right]
  $$
  measures the asymptotic growth of the number of holomorphic pluricanonical forms,
 while the numerical dimension  
 $
 \nu=\nu(X)=\sup \{ k \geq 0, c_1(K_X)^k \neq 0 \}.
 $
 measures the asymptotic growth of volumes under the NKRF,
 \begin{equation} \label{eq:volfkrn}
 {\rm Vol}_{\omega_t}(X)=
\left( \begin{array}{c} n \\ \nu \end{array} \right) 
 c_1(K_X)^{\nu} \{\omega_0\}^{n-\nu} e^{-(n-\nu)t} [1+o(1)].
  \end{equation}
  It is known that $\nu(X) \geq {\rm kod}(X)$ and the equality turns out to be equivalent to the abundance conjecture
 (see \cite[Conjecture 6.3]{Tos18}).
 The extra assumption made in \cite[Theorem 2.3]{GPSS22} requires that ${\rm kod}(X) \neq -\infty$;
it is equivalent to 
$$
\nu(X) \geq 0 \stackrel{?}{\Longrightarrow} {\rm kod}(X) \geq 0,
$$
which has been an open problem for the last fifty years.
   
   \begin{proof}
 Fix $\chi$ a smooth closed $(1,1)$-form representing $c_1(K_X)$. It follows from the $\partial\overline{\partial}$-lemma
 that $\omega_t=e^{-t}\omega_0+(1-e^{-t}) \chi+dd^c \f_t$ for some $\f_t \in {\mathcal C}^{\infty}(X)$. One can normalize
 the latter so that the NKRF is equivalent to the   parabolic equation
 $$
 (e^{-t}\omega_0+(1-e^{-t}) \chi+dd^c \f_t)^n=e^{\partial_t \f_t+\f_t-(n-\nu)t} \omega_0^n
 $$
 on $X \times \R^+$, with $\f_0=0$. Here $\nu$ denotes the numerical dimension of $K_X$.
 
 We claim that there exists $C_0,C_1>0$ such that for all $t>0$ and $x \in X$,
 $$
 \f_t(x) \leq C_0
 \; \; \text{ and } \; \; 
 \partial_t \f_t(x) \leq C_1.
 $$
 It will then follow from \eqref{eq:volfkrn} that $V_{\omega_t}^{-1}\omega_t^n=f_t \omega_0^n$ with
 $||f_t||_{\infty} \leq C_2$. The uniform diameter and non-collapsing estimates are thus consequences 
 of Theorem A.
 
 \smallskip
 
 \noindent {\it Upper bound on $\f_t$}.
 We set $V_t=\int_X \omega_t^n$ and  $I(t)=\int_X \f_t \frac{\omega_0^n}{V_0}$.
 It follows from \cite[Proposition 8.5]{GZbook} that 
 $\sup_X \f_t \leq I(t)+C$ for some uniform constant $C>0$, hence it suffices to bound
 $I(t)$ from above. The concavity of the logarithm ensures that  
 $$
 I'(t)=\int_X \partial_t \f_t \frac{\omega_0^n}{V_0} \leq -I(t)+
 [\log V_t+(n-\nu)t-\log V_0]\leq -I(t)-C'.
 $$
Using that $I(0)=0$ we conclude that $I(t) \leq C'$, as desired.
 
 \smallskip
 
 \noindent {\it Upper bound on $\partial\f_t$}.
 Consider $H(t,x)=(e^t-1) \partial_t \f_t(x)-\f_t(x)-h(t)$, where
 $h(t)=\nu t+(n-\nu)e^t$. A direct computation shows that
 $$
 \left( \frac{\partial}{\partial t}-\Delta_{\omega_t} \right)(H)=-{\rm Tr}_{\omega_t}(\omega_0)
 +(n-\nu)[e^t-1] +n-h'(t) \leq 0.
 $$
 The maximum principle thus ensures that $H(t,x) \leq \max_{x \in X} H(0,x) \leq 0$, hence
 $$
 (e^t-1) \partial_t \f_t(x) \leq \f_t(x)+\nu t +(n-\nu)e^t \leq C_0+ne^t.
 $$
 Thus $\partial_t \f_t(x) \leq C_1$ for all $t \geq 1$, while such an upper bound is
 clear on $X \times [0,1]$ by compactness.
  The proof is complete.
   \end{proof}

   \subsection{Fiberwise Calabi-Yau  metrics} \label{sec:ke}
   
   Our estimates can be applied in
 the study of adiabatic limits  of Ricci-flat K\"ahler metrics on a Calabi-Yau manifold 
 under the degeneration of the K\"ahler class, as initiated by Tosatti in \cite{Tos10}.
   
   Let $(X,\beta_X)$ be an $N$-dimensional K\"ahler manifold with nowhere vanishing holomorphic volume form $\Omega$
   normalized so that $\int_X i^{N^2} \Omega \wedge \overline{\Omega}=1$.
 Let $\pi:X \rightarrow Y$ be a holomorphic fibration onto a Riemann surface $(Y,\beta_Y)$, with connected fibres 
 by $X_y$, $y \in Y$.
We assume without loss of generality that $\int_{X_y} \beta_X^{N-1}=1$ and $\int_Y \beta_Y=1$,
and that the singular fibres   have at worst canonical singularities.
We let $\omega_t$ denote the Calabi–Yau metrics on X in the class of
$\beta_t=t\beta_X+ \pi^*\beta_Y$, $t>0$. 

\smallskip

Understanding the behavior of $(X,\omega_t)$ as $t \rightarrow 0$ as been the subject of intensive studies
in the past decade, notably through collaborations of
Gross, Hein, Li, Tosatti, Weinkove, Yang and Zhang (see \cite{Tos20} and the references therein).
Theorem A allows one to provide an alternative proof of the main result of \cite{Li23}.
   
   \begin{theorem} \cite[Theorem 1.4]{Li23}
   There exists $C>0$ such that 
   $$
   {\rm diam} \left(X_y, \frac{\omega_t}{t} \right) \leq C
   $$
   for all $0<t \leq 1$ and for all $y \in Y\setminus S$.
   \end{theorem}
   
  Here  $S\subset Y$ denotes the discriminant locus; $\pi$ is a submersion over $Y\setminus S$ so that every fiber $X_y, y\in Y\setminus S$ is smooth.

   \begin{proof}
   Set $n=N-1$ and $\omega_y={{\frac{\omega_t}{t}}}_{|X_y}$.  
   Observe  that $V_{\omega_y}=\int_{X_y} \omega_y^n=\int_{X_y} \beta_X^{N-1}=1$.
   It follows from \cite[Proposition 2.3]{Li23} and \cite[Lemma 4.4]{DnGG23} that there exists $p>1$ such that
   $$
   \omega_y^n =f ({\beta_X}_{|X_y})^n 
   \; \; \text{ with } \; \; 
   \int_{X_y} f^p ({\beta_X}_{|X_y})^n  \leq B,
   $$
   for some constant $B>0$ independent of $t,y$.
   The conclusion follows therefore from Theorem \ref{thm:diameter}
 applied  to the family $(X_y,\omega_y)$ with $A=1$.
   \end{proof}

\end{document}